\documentclass[11pt]{article}
\usepackage{amsfonts}
\usepackage{subfigure} 
\usepackage{epsfig}
\usepackage{amsmath, amsthm}
\newcommand{\be}{\begin{eqnarray}}     	\newcommand{\ee}{\end{eqnarray}}

\newcommand{\vol}{\mathrm{Vol}}
\newcommand{\dist}{\mathrm{dist}}
\newcommand{\rem}{\mathrm{Rm}}
\newcommand{\ric}{\mathrm{Ric}}
\newcommand{\diam}{\mathrm{diam}}

\title{Compactness results for the K\"ahler-Ricci flow}
\author{Natasa Sesum}

\date{} 
\theoremstyle{plain} 
\newtheorem{dummy}{Dummy}

\theoremstyle{definition}
\newtheorem{definition}[dummy]{Definition}

\theoremstyle{plain} \newtheorem{corollary}[dummy]{Corollary}
 \newtheorem{lemma}[dummy]{Lemma}
\newtheorem{theorem}[dummy]{Theorem}
\newtheorem{proposition}[dummy]{Proposition}

\newtheorem{claim}[dummy]{Claim} 

\begin{document}

\maketitle

\begin{abstract}
We consider the K\"ahler-Ricci flow $\frac{\partial}{\partial t}g_{i\bar{j}} =
g_{i\bar{j}} - R_{i\bar{j}}$ on a compact K\"ahler manifold $M$ with 
$c_1(M) > 0$, of complex dimension $k$. 
 We prove the $\epsilon$-regularity lemma for the K\"ahler-Ricci
flow, based on Moser's iteration. Assume that the Ricci curvature 
and $\int_M |\rem|^k\, dV_t$ are uniformly bounded along 
the flow. Using the $\epsilon$-regularity
lemma we derive the compactness result for the K\"ahler-Ricci flow.
Under our assumptions, if $k \ge 3$ in addition, using the compactness result
we show that  $|\rem| \le C$ holds
uniformly along the flow. This means the flow does not
develop any singularities at infinity. We use some  ideas of Tian from \cite{Ti}
to prove the smoothing property in that case.  
\end{abstract}

\section{Introduction}

There has been a lot of interest in compactness theorems for Riemannian 
manifolds under different geometric assumptions (see e.g. \cite{anderson1989},
\cite{bando1989}, \cite{tian1990}, \cite{TV}, \cite{natasa2005}). For example, in the
first three references the authors showed independently that if $(M_i,g_i)$
is a sequence of Einstein manifolds of real dimension $n$ ($n = 2k$), such that
\begin{enumerate}
\item[(i)]
$\diam(M_i,g_i) \le C$, 
\item[(ii)]
$\vol_{g_i}(M_i) \ge \delta$ and 
\item[(iii)]
$\int_{M_i}|\rem|^k\,dV_{g_i} \le C$, 
\end{enumerate}
for uniform constants $C, \delta$, then there is a subsequence of $(M_i,g_i)$
converging to a K\"ahler Einstein orbifold $(M_{\infty}, g_{\infty})$, with finitely many isolated
singularities. Let us denote by $\mathcal{K}_+(C,\delta,n)$ and $\mathcal{K}_-(C,\delta,n)$
the sets of all K\"ahler Einstein manifolds $(M,g)$ satisfying conditions (i),
(ii), (iii), with $\ric(g) = w_g$ and $\ric(g) = -w_g$, respectively. In \cite{Ti} it
was proved that both, $\mathcal{K}_+$ and $\mathcal{K}_-$ are compact
for $n \ge 3$, that is, $(M_{\infty}, g_{\infty})$ is a smooth K\"ahler manifold and 
the convergence is smooth everywhere.  The main tool in this smoothing property
is Kohn's estimate for $\bar{\partial}$-operators on strictly pseudoconvex  CR-manifolds.
 
Especially after the breakthrough Perelman (\cite{perelman2002}) made in using the Ricci
flow to complete Hamilton's program in proving the geometrization conjecture, there has been
an increasing interest in studying the K\"ahler Ricci flow as well. Let $(M,g)$ be a K\"ahler
manifold with $c_1(M) > 0$. Then we can assume $c_1(M) = [\ric(g)]$ is given by the K\"ahler
form $w_g$ and the K\"ahler Ricci flow is a solution to
\begin{equation}
\label{equation-kr}
\begin{split}
\frac{\partial}{\partial t} g_{i\bar{j}} &= g_{i\bar{j}} - R_{i\bar{j}} = \partial_i\bar{\partial}_j u, \\
g_{i\bar{j}}(0) &= g_{i\bar{j}}.
\end{split}
\end{equation}
Notice that the stationary solution to (\ref{equation-kr}) is a K\"ahler Einstein metric 
in $\mathcal{K}_+$. The generalization of  K\"ahler Einstein metrics are  K\"ahler Ricci
solitons.

\begin{definition}
We will say that a solution $g(t)$ to (\ref{equation-kr}) is a K\"ahler Ricci soliton if it moves
by one parameter family of automorphisms $\phi(t)$, or equivalently satisfies the equation
$$R_{i\bar{j}} + \rho g_{i\bar{j}} = \mathcal{L}_X(g),$$
where $X$ is a holomorphic vector field induced by automorphisms $\phi(t)$.
We say $g(t)$ is a gradient K\"ahler Ricci soliton if $X = \nabla f$ for a smooth function
$f$ that satisfies $f_{ij} = 0$. 
\end{definition}  
 
In \cite{natasa2005} we showed the analogous compactness result for the set of gradient
K\"ahler Ricci solitons satisfying similar geometric conditions to those in (i), (ii) and (iii). 
The next interesting problem arises in understanding the parabolic version of
compactness results, that is, the compactness results regarding the solutions to
(\ref{equation-kr}). This is tightly related to questions of converging K\"ahler Ricci
flows and understanding the possible limits.  

In \cite{cao1985} Cao proved that a solution to (\ref{equation-kr}) exists forever, that is,
a singularity does not occur at finite time.  It is very likely to happen that  
the flow develops singularities at infinity. A lot of progress has been made in studying the
limits of the K\"ahler Ricci flow. It is especially important to know
when we can expect to get K\"ahler Einstein metrics in a limit.  Among
first works on that topic is the work by Chen and Tian
(\cite{chen2002}, \cite{chen2001}) where they proved that if $M$
admits a K\"ahler Einstein metric with positive scalar curvature and
if an initial metric has a nonnegative bisectional curvature that is
positive at least at one point, then the K\"ahler Ricci flow converges
exponentially fast to the K\"ahler Einstein metric with constant
bisectional curvature. After the work of Perelman (\cite{perelman2002})
appeared,
in \cite{CCZ} it was proved 
that if the bisectional curvature $R_{i\bar{i}j\bar{j}}$ is positive at time $t=0$ then
the curvature operator stays uniformly bounded along the flow. 
In \cite{TiZh} it has been proved that if (in any complex dimension)
$M$ is a compact K\"ahler manifold which admits a K\"ahler-Ricci soliton $(g_{RS},X)$ defined
by a vector field $X$, then any solution $g(\cdot,t)$ of (\ref{equation-kr}) converges to $g_{RS}$,
if the initial K\"ahler metric $g_0$ is invariant under $K_X$, an $1$-parameter
subgroup of a group of automorphisms of $M$ generated by $X$.
In the recent preprint \cite{PSSW},
the authors have proved the convergence of the K\"ahler-Ricci flow to a
K\"ahler-Einstein metrics under certain assumptions.
By the result in
\cite{natasa2004}, if the curvature is uniformly bounded along the flow then it
sequentially converges to K\"ahler-Ricci solitons. 
In \cite{natasa_KR} we have proved  that if $|\ric(g(t))| \le C$ along the flow, for a uniform constant
$C$, then for any sequence $t_i\to \infty$ there is a subsequence so that
$(M,g(t_i+t)) \to (M_{\infty},g_{\infty}(t))$, where $M_{\infty}$ is
smooth outside a closed singular set $S$ which is of codimension at least
$4$ and the convergence is smooth outside $S$. 

Due to Perelman (see also \cite{se_ti}) we have the following uniform
estimates for the K\"ahler Ricci flow: there are uniform constants $C$
and $\kappa$ such that for all $t$,
\begin{enumerate}
\item
$|u(t)|_{C^1} \le C$,
\item
$\diam(M, g(t)) \le C$,
\item
$|R(g(t))| \le C$,
\item
$(M,g(t))$ is $\kappa$-noncollapsed.
\end{enumerate}
This together with the uniform lower bound on Ricci curvatures
along the flow gives a uniform upper bound on the Sobolev constant, that is, there is a
uniform constant $C_S$ so that for any $v\in C_0^1(M)$ we have that
\begin{equation}
\label{equation-Sobolev}
(\int_M v^{\frac{4n}{2n-2}}dV_{g(t)})^{\frac{2n}{2n-2}} \le C_S
\int_M|\nabla v|^2 dV_{g(t)},
\end{equation} 
for all times $t \ge 0$. This enables us to work with integral estimates.
By the recent results in \cite{Ye} and \cite{Z}, that have been obtained independently,
we have that due to Perelman's  uniform upper bound on the scalar curvature,
(\ref{equation-Sobolev}) holds uniformly along the flow without any geometric
assumptions on the curvatures.

In this paper we first show the following compactness result for the 
K\"ahler-Ricci flow.

\begin{theorem}
\label{theorem-convergence}
Let $g(t)$ be the K\"ahler Ricci flow on a compact, K\"ahler manifold
$M$, with $c_1(M) > 0$, with Ricci curvatures uniformly bounded and
with $\int_M|\rem(g(t))|^{n/2} dV_{g(t)} \le C$ along the flow. Then
for every sequence $t_i\to\infty$ there is a subsequence so that
$(M,g(t_i+t))$ converges to $(M_{\infty},g_{\infty}(t))$, where
\begin{enumerate}
\item[(a)]
$M_{\infty}$ is an orbifold with finitely many isolated singular
points, $\{p_1,\dots,p_N\}$, and the convergence is smooth outside
those singular points.
\item[(b)]
The limit metric $g_{\infty}$ is a K\"ahler Ricci soliton in an
orbifold sense, that is, satisfies the K\"ahler Ricci soliton equation,
\begin{eqnarray}
\label{equation-kahler}
(g_{\infty})_{i\bar{j}} - R_{i\bar{j}}(g_{\infty}) &=&
\partial_i\partial_{\bar{j}}f_{\infty}, \\
\partial_{\bar{i}}\partial_{\bar{j}}f_{\infty} &=& 0 \nonumber,
\end{eqnarray}
off the singular points. Moreover, for every singular point $p_j$,
there is a neighbourhood in $M_{\infty}$ which lifts to an open set
$D_j \subset  \mathrm{C}^{n/2}$, and the lifting of an orbifold
metric $g_{\infty}$ satisfies equivalent equations to (\ref{equation-kahler})
in $D_j$.
\end{enumerate}
\end{theorem}

For a sequence $(M,g(t_i+t))$ converging to $(M_{\infty},g_{\infty})$
in the sense described by the previous theorem in (a) and (b),
we will say it {\em converges to a K\"ahler Ricci soliton in an orbifold sense}.
We will use Theorem \ref{theorem-convergence} to prove the 
following smoothing theorem
that shows the flow does not develop any singularities at infinity in the case
$k  > 3$.

\begin{theorem}
\label{thm-curv-bound}
Let $g(t)$ be a K\"ahler Ricci flow on a compact, K\"ahler manifold
$M$ of complex dimension $k$ ($n = 2k$), with $c_1(M) > 0$, Ricci curvatures 
uniformly bounded and
with $\int_M|\rem(g(t))|^{n/2} dV_{g(t)} \le C$ along the flow. Then if
$k\ge 3$, the curvature 
operator is uniformly bounded along the flow.  
\end{theorem}

The proof of the later theorem relies on Kohn's estimate for $\bar{\partial}$-operator 
that is proved in \cite{Ti} and Perelman's pseudolocality theorem (\cite{perelman2002}).
The estimate from \cite{Ti} works only if $n \ge 6$. We expect that in the case of 
surfaces ($n = 4$) the curvature stays uniformly bounded along the flow as well. 
It is known by classification theory of complex surfaces that 
$\mathbb{CP}^2\sharp q\bar{\mathbb{CP}}^2$
($0\le q \le 8$) and $\mathbb{CP}^1\times \mathbb{CP}^1$ are the only compact $4$-manifolds
which admit  complex structures with positive first Chern class.  By the result
of Tian in \cite{tian1990} it is known that each of the surfaces 
$\mathbb{CP}^1\times\mathbb{CP}^1$ and
$\mathbb{CP}^2\sharp q\bar{CP}^2$ admits a K\"ahler Einstein metric
for $q = 0$ and $3 \le q \le 8$. Cao (\cite{C}) and Koiso (\cite{K}) independently showed that
$\mathbb{CP}^2\sharp\bar{\mathbb{CP}}^2$ admits a nontrivial K\"ahler Ricci soliton metric.
Wang and Zhu (\cite{WZ}) showed the same result later for
$\mathbb{CP}^2\sharp2\bar{\mathbb{CP}}^2$.Up to the invariance of the initial metric under $K_X$,
by \cite{TiZh}, we have that for any K\"ahler Ricci flow on a compact K\"ahler
surface $M$ with $c_1(M) > 0$, the curvature operator stays uniformly bounded along 
the flow. Moreover the flow always converges.

The organization of the paper is as follows. In section \ref{section-eps} we will 
show the $\epsilon$-regularity for the K\"ahler-Ricci flow. The proof will be based
on standard Moser's iteration argument. In section \ref{section-compactness}
we will prove the compactness result for the K\"ahler-Ricci flow, Theorem
\ref{theorem-convergence}, using the $\epsilon$-regularity lemma established
in the previous section and some ideas of \cite{anderson1989}, \cite{bando1989}
and \cite{tian1990}. In section \ref{section-tian} we will show that some of
the results from \cite{Ti} regarding the convergence of plurianticanonical 
divisors apply also to our case. In section \ref{section-smoothing}
we will prove   Theorem \ref{thm-curv-bound},
that is, the smoothing property of the compactness result in the case $k \ge 3$, 
which yields the uniform curvature bound along the flow.  
In other words,  we will show we have the smooth convergence in
Theorem \ref{theorem-convergence}. 

{\bf Acknowledgements:} The author would like to thank her advisor G.Tian
for useful discussions. Many thanks to Valentino Tosatti  who brought 
the paper \cite{Ti} to author's attention.

\section{An $\epsilon$-regularity lemma for the K\"ahler-Ricci flow}
\label{section-eps}

In this secton we will establish the $\epsilon$-regularity lemma that
is a parabolic analogue of the $\epsilon$-regularity lemma in the elliptic
setting, that appeared independently  in \cite{anderson1989}, \cite{bando1989}
and \cite{tian1990}. It will allow us obtain pointwise 
curvature bounds away from finitely many curvature concentration points.
More precisely, the obtained quadratic 
curvature decay  will be useful in understanding the orbifold structure
of a limit of the K\"ahler Ricci flow when $t\to\infty$. 
For the proof of the $\epsilon$-regularity lemma we will need Moser's weak
maximum principle argument in a form proved my Dai, Wei and Ye in
\cite{dai1996}.  Along these lines, the $\epsilon$-regularity lemma
for the mean curvature flow has been established in \cite{ecker1995}.

\begin{theorem}[Moser's weak maximum principle]
\label{theorem-theorem_Moser}
Let $f,b$ be smooth nonnegative functions on $N\times [0,T]$ wich
satisfy the following inequality
$$\frac{d}{dt}f \le \Delta f + bf,$$
on $N\times[0,T]$. If $b$ satisfies
$$\sup_{t\in[0,T]}(\int_Nb^{q/2})^{2/q} \le \beta,$$ 
for some $q > n$. Let
$C_S = \max_{[0,T]}C_S(g(t))$ and $l = \max_{[0,t]}||\frac{dg}{dt}||_{C^0}$. 
Then given any $p_0$ there exists a constant $C = C(n,q,p_0,C_S,l,T,R)$
such that for $x\in M\times (0,T]$
$$|f(x,t)| \le Ct^{-\frac{n+2}{2p_0}}(\int_0^T\int_{B_R}f^{p_0})^{1/p_0},$$
where $B_R$ is a geodesic ball of radius $R$, defined in terms of metric $g(0)$.
 \end{theorem}

Denote by $f = |\rem|$. By the evolution equation of the curvature
operator we have that $f$ satisfies
\begin{equation}
\label{equation-equation_ev_curv}
\frac{d}{dt}f \le \Delta f + Cf^2,
\end{equation} 
for some uniform constant $C$. Assume the following condition, that we will refer to as to
a {\it euclidean volume growth}. 
\begin{equation}
\label{eq-new-cond}
\vol_t B_t(p,r) \le Cr^n, \,\, \mbox{for all} \,\, t\ge 0 \,\, p\in M \,\, 
\mbox{and} \,\, 0 \le r \le \diam(M,g(t)).
\end{equation}
Assuming the euclidean volume growth condition, we have the following proposition.

\begin{proposition}
\label{lemma-lemma_regularity}
There are $C = C(n,C_S)$, $\epsilon = \epsilon(n,C_S)$ and $0 < \xi_1, \xi_2 < 1$, 
such that if for some $t > 0$ and 
$\delta \le r^2$ (we may assume $\delta$ is comparable with $r^2$),
\begin{equation}
\label{equation-equation_condition}
\sup_{s\in [0,\delta]}\int_{B_{t}(p,3r)}f^{n/2}dV_{t+s} < \epsilon,
\end{equation}
then
\begin{equation}
\label{equation-equation_conclusion}
\sup_{B_t(p,\xi_1 r)\times[\xi_2\delta, \delta]}|f| \le C(\frac{1}{r^2} +
\frac{1}{\delta})\sup_{s\in
[0,\delta]}(\int_{B_{t}(p,2r)}f^{n/2})^{2/n} dV_{t+s},
\end{equation}
where $p$ is any point in $M$ and our estmates are independent of the
initial time $t$.
\end{proposition}

By the independent 
results of Zhang (\cite{Z}) and Ye (\cite{Ye}), there are positive
constants $A$ and $B$, depending only on the geometric quantities of $g(0)$,
so that for all $t \in [0,\infty)$ and all $u\in W^{1,2}(M)$,
$$\left(\int_M|u|^{\frac{2n}{n-2}}\,dV_t\right)^{\frac{n-2}{n}} \le 
A\int_M (|\nabla u|^2 + \frac{R}{4}u^2)\, dV_t + B\int_M u^2\, dV_t.$$
Due to Perelman, $|R| \le C$ along the flow and the above inequality becomes
\begin{equation}
\label{eq-sob-new}
\left(\int_M|u|^{\frac{2n}{n-2}}\,dV_t\right)^{\frac{n-2}{n}} \le 
A\int_M |\nabla u|^2 + B_1\int_M u^2\, dV_t,
\end{equation}
for a different constant $B_1$.

We will first give a proof of Proposition \ref{lemma-lemma_regularity} under the assumption that
$|\ric| \le C$ for all $t \ge 0$. Later we will show how we can replace the Ricci curvature condition 
with the euclidean volume growth condition.

Let $\eta$ be a cut off function, so that
$\eta = 1$ in $B_t(p,r)$ and $\eta = 0$ outside the ball $B_t(p,2r)$ and
$|\nabla\eta|^2(t) \le \frac{C}{r^2}$. Since,
$$\frac{d}{dt}|\nabla\eta|^2 = |\nabla \eta|^2 - \ric(\nabla\eta,\bar{\nabla}\eta),$$
and therefore $|\frac{d}{dt}\ln |\nabla\eta|^2| \le C$, it easily follows that,
\begin{equation}
\label{equation-eta}
\sup_{s\in [t-r^2,t+r^2]}|\nabla\eta|^2(s) \le \frac{\tilde{C}}{r^2},
\end{equation}
for a uniform constant $\tilde{C}$.

\begin{proof}[Proof of Proposition \ref{lemma-lemma_regularity}]
Since the equations (\ref{equation-equation_ev_curv}), (\ref{equation-equation_condition})
and (\ref{equation-equation_conclusion}) are scale invariant, we may
assume the radius of the ball is $r=1$. In other words, by rescaling
our metric by  factor $r^{-2}$, that is $\bar{g}(t) = r^{-2}g(r^2t)$, it is
enough to show that there are uniform constants $C, \xi_1$ and  $\xi_2$ so that
$$\sup_{B_t(p,\xi_1)\times[\xi_2\delta r^{-2}, \delta r^{-2}]}|f|(x,t+s) \le C\sup_{s\in
[0,\delta r^{-2}]}(\int_{B(p,2)}f^{\frac{n}{2}} dV_{t+s})^{2/n},$$ 
for some uniform constant
$C$, if $\sup_{s\in [0,\delta r^{-2}]}\int_{B_t(p,3)} f^{n/2}dV_{t+s} < \epsilon$ 
(where $f(t) = |\rem|(\bar{g}(t))$).
Multiply (\ref{equation-equation_ev_curv}) by $\eta^2f$, where $\eta$
is a cut off function with compact support in $B_t(p,3)$,
which is equal to $1$ on $B_t(p,2)$ and such that
$|\nabla\eta| \le C$ ($|\nabla\eta|$
is computed in norm $\bar{g}(t)$). Furthermore, since $|\ric|(\bar{g}(t)) \le Cr^2$,
for all $t$, we have
$$\frac{d}{ds}|\nabla\eta|^2(t+s) \le Cr^2|\nabla\eta|^2(t+s),$$ 
and $|\nabla\eta|(t+s) \le \tilde{C}$ for $s\in [0,\delta r^{-2}]$.  After multiplying
(\ref{equation-equation_ev_curv}) by $\eta^2f$ and integrating by
parts we get
\begin{equation}
\label{equation-equation_eq1}
\int\eta^2|\nabla f|^2 - 2\int\eta f|\nabla\eta||\nabla f| +
\frac{1}{2}\frac{d}{ds}\int f^2\eta^2 \le C\int\eta^2f^2 + C\int\eta^2f^3.
\end{equation}
By H\"older,  Sobolev and interpolation inequalities we have
\begin{equation}
\label{equation-equation_eq2}
\int\eta f|\nabla f||\nabla\eta| \le \frac{1}{2}\int \eta^2|\nabla f|^2 +
C\int|\nabla\eta|^2f^2,
\end{equation}
\begin{eqnarray}
\label{equation-equation_eq3}
\int\eta^2f^3 &=& \int (\eta f^2)\cdot f \le 
(\int_{B_t(p,3)}f^{n/2})^{2/n}(\int (\eta f)^{\frac{2n}{n-2}})^{\frac{n-2}{n}} \nonumber \\
&\le& \epsilon C_S (\int|\nabla f|^2\eta^2) + \int |\nabla\eta|^2 f^2).
\end{eqnarray}
Combining  (\ref{equation-equation_eq1}), (\ref{equation-equation_eq2})
and (\ref{equation-equation_eq3}) and choosing $\epsilon$ small so
that we can absorb $\epsilon C_S (\int|\nabla f|^2\eta^2)$ into the
left hand side of (\ref{equation-equation_eq1}) yield
\begin{equation}
\label{equation-equation_extension}
\frac{d}{ds}\int\eta^2 f^2 + C\int|\nabla f|^2\eta^2 \le C\int f^2(\eta^2 + |\nabla\eta|^2).
\end{equation}
This implies
\begin{equation}
\label{equation-equation_est_nab0}
\begin{split}
\int_0^{\delta r^{-2}}\int_{B_t(p,2)}|\nabla f|^2 &\le C(\int_0^{\delta r^{-2}}\int_{B_t(p,3)} f^2\, 
dV_{t+s}\, ds + \int_{B_t(p,3)} f^2 dV_t)  \\
&\le \tilde{C},
\end{split}
\end{equation}
for a uniform constant $\tilde{C}$, since 
\begin{equation}
\label{equation-n2}
\int_M f^2 \, dV_t \le (\int_M f^{\frac{n}{2}}\, dV_t)^{4/n}\cdot (\vol_t(M))^{(n-4)/n} \le C,
\end{equation}
for a uniform constant $C$.
We also have the evolution equation for $|\nabla f|^2$,
\begin{equation}
\label{equation-equation_ev_nabla}
\frac{d}{ds}|\nabla f|^2 \le \Delta |\nabla f|^2 - 2|\nabla^2 f|^2 + C|\nabla f|^2f.
\end{equation}
Let $\eta$ be a cut off function with compact support in
$B_t(p,2)$ such that it is identically $1$ on
$B_t(p,1)$. Multiply (\ref{equation-equation_ev_nabla}) by $\eta^2$
and integrate it over $M$.
\begin{equation}
\label{equation-equation_nab1}
\frac{d}{ds}\int\eta^2|\nabla f|^2 \le C\int\eta^2|\nabla f|^2 -
2\int\eta \nabla|\nabla f||\nabla f|\nabla\eta - \int\eta^2|\nabla^2 f|^2 +
C\int\eta^2|\nabla f|^2 f.
\end{equation}
Similarly as before we can estimate
\begin{equation}
\label{equation-equation_nab2}
2\int\eta \nabla|\nabla f|\nabla\eta|\nabla f| \le
\frac{1}{3}\int\eta^2|\nabla|\nabla f||^2 + C\int|\nabla \eta|^2|\nabla f|^2,
\end{equation}
\begin{eqnarray}
\label{equation-equation_nab3}
\int \eta^2|\nabla f|^2f &\le& (\int f^{n/2})^{2/n}(\int (|\nabla f|\eta)^{\frac{2n}{n-2}})^{\frac{n-2}{n}} \nonumber\\
&<& \epsilon C_S(\int|\nabla|\nabla f||^2\eta^2 + \int|\nabla f|^2|\nabla \eta|^2).
\end{eqnarray}
Combining (\ref{equation-equation_nab1}),
(\ref{equation-equation_nab2}) and (\ref{equation-equation_nab3}),
choosing $\epsilon$ small enough and using that $|\nabla|\nabla f||
\le |\nabla^2 f|$ yield
\begin{eqnarray}
\label{equation-equation_psi}
\frac{d}{ds}\int|\nabla f|^2 \eta^2 &\le& C\int |\nabla f|^2(\eta^2 + |\nabla\eta|^2) \nonumber \\
&\le& C\int_{B_t(p,2)}|\nabla f|^2.
\end{eqnarray}
As in \cite{dai1996} choose a cut off function $\psi(s)$ such that
$\psi(s)$ is $0$ for $0\le t \le \delta r^{-2}/4$, 
$\frac{4}{\delta r^{-2}}(t-\delta r^{-2}/4)$ for $t\in [\delta r^{-2}/4,\delta r^{-2}/2]$
and $1$ for $t\in [\delta r^{-2}/2, \delta r^{-2}]$. Multiplying (\ref{equation-equation_psi})
by $\psi(s)$ we obtain
$$\frac{d}{ds}(\psi\int|\nabla f|^2\eta^2) \le (C\psi + \psi')\int |\nabla f|^2\eta^2
+ C\psi\int f^2|\nabla\eta|^2.$$
Integrating  in $s$ we obtain
$$\int\eta^2|\nabla f|^2(t+s) \le C\int_0^{\delta r^{-2}}\int_{B_t(p,2)}|\nabla f|^2dV_{t+s}ds +
C\int_0^{\delta r^{-2}}\int_{B_t(p,3)}f^2 dV_{t+s}ds \le \tilde{C},$$
for a uniform constant $\tilde{C}$, independent of $t$ and 
$s\in [\delta r^{-2}/2,\delta r^{-2}]$. By (\ref{equation-equation_est_nab0}) we get
$$\sup_{[\delta r^{-2}/2,\delta r^{-2}]}\int_{B_t(p,1)}|\nabla f|^2 \le 
C(\int_0^{\delta r^{-2}}\int_{B_t(p,3)} f^2\, dV_{t+s}\, ds + \int_{B_t(p,3)} f^2\, dV_t),$$
for a uniform constant $C$.
By Sobolev inequality we get
\begin{equation}
\label{equation-equation_better}
\begin{split}
\sup_{[\delta r^{-2}/2,\delta r^{-2}]}(\int_{B_t(p,1)}f^{\frac{2n}{n-2}}dV_{t+s})^{(n-2)/n} &\le 
\tilde{C}_1(\int_0^{\delta r^{-2}}\int_{B_t(p,3)} f^2 dV_{t+s}\, ds + \\
& \int_{B_t(p,3)} f^2\, dV_t).
\end{split}
\end{equation}
In other words, the $L^{\frac{2n}{n-2}}$-norm of $|f|$ can be estimated in terms of the
$L^2$-norm of $|f|$. Similarly, by multiplying (\ref{equation-equation_ev_curv}) by
$\eta^2 f^{\frac{n+1}{n-1}}$ and proceeding as above, we have
\begin{eqnarray*}
& &\sup_{s\in [\delta r^{-2}(\frac{1+1/2}{2}),\delta r^{-2}]} (\int_{B_t(p,1/3)} f^{2(n/(n-2))^2}\, 
dV_{t+s})^{(n-2)/n} \le \\
& & C_1(\int_{\frac{\delta r^{-2}}{2}}^{\delta r^{-2}}\int_{B_t(p,1)}  f^{\frac{2n}{n-2}}\, dV_{t+s}\, ds + 
\int_{B_t(p,3)} f^{\frac{2n}{n-2}} dV_{t+\delta r^{-2}}) \le C.
\end{eqnarray*}
Continuing the above process $l$ times, so that $(\frac{n}{n-2})^l \ge n$, using
(\ref{equation-n2}) we obtain
\begin{equation}
\label{equation-equation_better}
\begin{split}
& \sup_{s\in [\delta r^{-2}(\frac{\sum_{j=0}^l 2^{-j}}{2}),\delta r^{-2}]} (\int_{B_t(p,1/3^l)} f^{2(n/(n-2))^l}\, dV_{t+s})^{((n-2)/n)^l} \\
&\le C(l) \sup_{s\in [0,\delta r^{-2}]}\int_M f^2 \, dV_s  
\le \tilde{C}(l) \sup_{s\in [0,\delta r^{-2}]}\left(\int_M f^{\frac{n}{2}} \, dV_s\right)^{\frac{4}{n}}. 
\end{split}
\end{equation}
By Theorem \ref{theorem-theorem_Moser}, its proof and by the estimate
(\ref{equation-equation_better}) we get
\begin{eqnarray*}
\sup_{B_t(p,1/(2\cdot 3^l))\times [ \delta r^{-2}\left(\frac{\sum_{j=0}^l 2^{-j}}{2}\right)
,\delta r^{-2}]} |f|(x,t)
&\le& C(l)\sup_{s\in [0,\delta r^{-2}]}\left(\int_{B_t(p,1)}f^2\right)^{1/2} \\
&\le& \tilde{C}(l)\sup_{s\in [0,\delta r^{-2}]}\left(\int_{B_t(p,1)}f^{\frac{n}{2}}\right)^{\frac{2}{n}}.
\end{eqnarray*}
If we rescale metric $\bar{g}$ back to our original metric $g(t)$ we get
(\ref{equation-equation_conclusion}) by putting $\xi_1 := \frac{1}{2\cdot 3^l}$ and
$\xi_2 := \left(\frac{\sum_{j=0}^l 2^{-j}}{2}\right)$.
\end{proof}

To show how we can remove the hypothesis on the Ricci curvature bound,
note that the main tool in the proof of Proposition \ref{lemma-lemma_regularity} is the
Sobolev inequality (\ref{equation-Sobolev}) that we have due to \cite{Ye} and \cite{Z}.    
We have used the Ricci curvature bound to 
control the norms $\sup_{s\in [-r^2,r^2]}|\nabla\eta|_{t+s}$, where $\eta$ 
is the cut off function  with a support in $B_t(p,r)$, independently of
$t$ and $s \in [-r^2,r^2]$. We can overcome 
that difficulty by choosing the cut off function differently, if we assume the euclidean volume growth condition. The difference is that we will let the cut off function evolve in time as the
metric evolves, unlike in the situation above where the cut off function was 
time independent. 

\begin{corollary}
\label{cor-regularity}
Proposition \ref{lemma-lemma_regularity} still holds if we replace  the condition
on a uniform bound of Ricci curvatures along the flow with a weaker
condition (\ref{eq-new-cond}).
\end{corollary}

\begin{proof}
Let $\eta_0$ be a cut off function as in Proposition \ref{lemma-lemma_regularity},
so that $|\nabla \eta_0|_t^2 \le \frac{C}{r^2}$,
where the norm of $\nabla\eta_0$ is taken with respect to metric $g(t)$.
Lets evolve $\eta(s)$ for $s\in [0,\delta]$, in a ball $B_t(p,2r)$ as follows,
(for our purposes we can assume $\delta$ being comparable with $r^2$)
\begin{equation}
\label{eq-eta}
\begin{split}
\frac{\partial}{\partial s}\eta &= \Delta\eta,\\
\eta(x,0) &= \eta_0(x), \\
\eta(x,s) &= 0, \,\, \mbox{for} \,\, (x,s) \in \partial B_t(p,2r)\times [0,\delta],
\end{split}
\end{equation}
where $\Delta$ is taken with respect to a changing metric $g(t+s)$.
The evolution equation for $\nabla\eta$ is
$$\frac{\partial}{\partial s}|\nabla\eta|^2 = \Delta|\nabla\eta|^2 - |\nabla\nabla \eta|^2
- |\nabla\bar{\nabla}\eta|^2,$$
where $|\nabla\eta|^2$ is taken with respect to $g(t+s)$. By the maximum 
principle it follows 
\begin{equation}
\label{eq-nabla-eta}
|\nabla\eta|^2(\cdot,s) \le \max_{B_t(p,2r)}|\nabla\eta_0|^2 \le \frac{C}{r^2},
\end{equation}
$$\eta(x,s) \ge 0, \qquad \mbox{for all} \,\, s\in [0,\delta],$$
$$\max_{B_t(p,2r)\times [0,\delta]}|\eta|(x,s) \le 1.$$
By the results in \cite{Ye} and \cite{Z} we have Sobolev inequlity (\ref{eq-sob-new})
holding uniformly along the flow. In the following claim we will show $\eta$ has
a uniform lower bound and therefore we can apply almost the same proof of
Proposition \ref{lemma-lemma_regularity} to get the result.

\begin{claim}
\label{claim-eta}
There exist  uniform constants $\delta_1, \delta_2 > 0$, independent of $t$ and $r > 0$, 
so that $\eta(x,t+s) \ge \delta_1 > 0$, for all $x\in B_t(p,\delta_2r)$ and all $s\in [0,\delta]$. 
\end{claim}

\begin{proof}
All we need to show is that $\eta(p,t+s) \ge \alpha$ for all $s\in [0,\delta]$ and a uniform
constant $\alpha$, independent of $t$, $p\in M$ and $r > 0$. This together with 
(\ref{eq-nabla-eta}) will yield the result.
If we integrate (\ref{eq-eta}) over $B_t(p,2r)$ we get
$$\frac{d}{ds}\int_{B_t(p,2r)}\eta(x,t+s)\, dV_{t+s} = \int_{B_t(p,2r)}\eta(x,t+s)(k - R)\, dV_{t+s},$$
which, since $\eta(\cdot,t+s) \ge 0$ implies
$$|\frac{d}{ds}\int_{B_t(p,2r)}\eta(x,t+s)\, dV_{t+s}| \le C\int_{B_t(p,2r)}\eta(x,t+s)\, dV_{t+s},$$
and therefore, since $\eta(x,t) = 1$ for $x\in B_t(p,r)$ and $\vol_t B_t(p,r) \ge \kappa r^n$,
\begin{eqnarray*}
\int_{B_t(p,2r)}\eta(x,t+s)\, dV_{t+s} &\ge& e^{-Cs} \int_{B_t(p,2r)}\eta(x,t)\,dV_t \ge 
C_1\int_{B_t(p,r)}\eta(x,t)\, dV_t \\
&\ge& C_1\kappa r^n = Ar^n, 
\end{eqnarray*}
for all $s\in [0,\delta]$, $p\in M$ and $r > 0$. By condition (\ref{eq-new-cond})
this implies
\begin{equation}
\label{eq-mv}
\frac{1}{\vol_{t+s} (B_t(p,2r))}\int_{B_t(p,2r)} \eta(x,t+s)\, dV_{t+s} \ge A_1,
\end{equation}
for a uniform constant $A_1$, independent of $r > 0$, $p\in M$ and $s\in [0,\delta]$.
Fix $s\in [0,\delta]$. 
By (\ref{eq-mv}) there exists at least one $q_1\in B_t(p,2r)$ so that
$\eta(q_1,t+s) \ge A_1$. By (\ref{eq-nabla-eta}) there exists a uniform constant
$\delta > 0$ so that $\eta(x,t+s) \ge A_1/2$ for all $x\in B_t(q_1,\delta r)$. If
$p\in B_t(q_1,\delta r)$ we are done. If not, let $r_1 = \dist_{t+s}(q_1,p)/4$ and
consider the ball $B_t(p, 2r_1)$. Radius $r_1$ will play the role of $r$ above. Since
$q_1$ is not in $B_t(p,2r_1)$, by the same arguments as before we can find $q_2 \in B_t(p,2r_1)$
so that $\eta(x,t+s) \ge A_1/2$ for all $x\in B_t(q_2,\delta r_1)$. This way we will
construct a sequence $\{q_j\}$ such that $q_j \to p$ as $j\to\infty$ and
$\eta(q_j,t+s) \ge A_1/2$. By the continuity of $\eta$, $\eta(p,t+s) \ge A_1/2 =: \alpha$.
Since $s$ was an arbitrary number in $[0,\delta]$, we are done.
\end{proof}

As we remarked earlier, having the uniform Sobolev inequality (\ref{eq-sob-new})
and Claim \ref{claim-eta} finishes the proof of Corollary \ref{cor-regularity}.
\end{proof}

\section{Convergence to a K\"ahler-Ricci soliton orbifold with finitely many
singulr points}
\label{section-compactness}

In this section we will prove the compactness result for the K\"ahler-Ricci flow,
which is dimension independent.

\begin{proof}[Proof of Theorem \ref{theorem-convergence}]
Take $\epsilon << \epsilon_0$  small, where $\epsilon_0$ is taken
from Proposition \ref{lemma-lemma_regularity}. Define 
\begin{eqnarray*}
D_i^r = \{x\in M\:\:\:
|\:\:\: \int_{B_{t_i}(p,2r)}|\rem|^{n/2}\, dV_{t_i} < \epsilon\}, \\
L_i^r = \{x\in M \:\:\:|\:\:\: \int_{B_{t_i}(x,2r)}|\rem|^{n/2}\, dV_{t_i} > \epsilon\}.
\end{eqnarray*}
By the covering argument, similarly as in \cite{natasa2005} we can get
there is a uniform upper bound on a number $N$ of points,
$\{x_{i1}^r,\dots, x_{iN}^r\}$, that are the centres of balls of radius $2r$
covering sets $L_i^r$, so that the corresponding concentric balls of
radius $r$ are disjoint. Notice that $N$ is independent of both $i$ and $r > 0$.
To see that, assume
$$\int_{B_{t_i}(x_{ij}^r,2r)}|\rem|^{n/2}dV_{t_i} \ge \epsilon.$$
By the $\kappa$-noncollapsing that was proved by Perelman (see \cite{se_ti}),
there is a uniform constant $\kappa > 0$ so that
\begin{equation}
\label{eq-kappa-non}
\vol_t B_t(p,r) \ge \kappa r^n, \qquad \mbox{for all} \,\, t > 0 \,\, \mbox{and all}
\,\, r > 0 \,\, p\in M.
\end{equation}
Since $\ric \ge -C$, by Bishop-Gromov comparison principle, for every
$r > \delta$ and every $p\in M$,
$$\frac{\vol_t B_t(p,r)}{V_{-C}(r)} \le \frac{\vol_t B_t(p,\delta)}{V_{-C}(\delta)},$$
where $V_{-C}(r)$ is the volume of a ball of radius $r$ in a space form of
constant sectional curvature $-C$. If we let $\delta\to 0$, the right hand side
of the previous inequality converges to a constant equal to the volume 
of a unit ball in $\mathbb{R}^n$. Call it $w_n$. Then,
\begin{equation}
\label{eq-vol-upper}
\frac{\vol_t B_t(p,r)}{r^n} \le w_n\frac{V_{-C}(r)}{r^n} \le C_1,
\end{equation}
for a uniform constant $C_1$ and  $r \le r_0$.
Notice that (\ref{eq-kappa-non}) and (\ref{eq-vol-upper}) imply there is a uniform upper
bound $m$ on the number of disjoint balls of radius $r$ contained in a ball of
radius $2r$.  Then,
\begin{eqnarray*}
N\epsilon &\le& \sum_j \int_{B_{t_i}(x_{ij}^r,2r)}|\rem|^{n/2} \, dV_{t_i}\\
&\le& m\int_M|\rem|^{n/2}dV_{t_i} \le  Cm = \tilde{C},
\end{eqnarray*}
which yields
\begin{equation}
\label{eq-upper-N}
N \le \frac{\tilde{C}}{\epsilon}.
\end{equation}
In \cite{natasa_KR} we have proved that given the K\"ahler-Ricci flow with
uniformly bounded Ricci curvatures, then for every sequence
$t_i\to\infty$ there exists a subsequence such that $(M,g(t_i + t))\to
(Y,\bar{g}(t))$. The convergence is smooth outside a singular set
$S$, which is closed and at least of codimension four. In \cite{se_ti}
we showed that $\bar{g}(t)$ solves the K\"ahler-Ricci soliton equation 
off $S$. 

\begin{proposition}
\label{prop-finite}
The closed set $S$ consists of finitely many points.
\end{proposition}

\begin{proof}
The proof goes by contradiction. Assume the proposition is false.
Since $S$ is a closed subset of $Y$, for every $r$ we can find 
a finite cover of $S$ with balls of radius $2r$ so that the corresponding 
concentric balls of radius $r$ stay disjoint. Choose $r > 0$ small so that
the number $L$ of above balls of radius $2r$ covering $S$ is bigger than 
$1000[\frac{\tilde{C}}{\epsilon}]$, where the constants are taken 
from (\ref{eq-upper-N}). We can always do that if $S$ is not just a 
set of finitely many isolated points. Denote by $\bar{p}_j$, for $1 \le j \le L$
the centers of those balls. Let $S(r) = \cup_{j=1}^L B_{\bar{g}}(\bar{p}_j,2r)$.
From \cite{natasa_KR} there are points $p_j^i \in M$ and 
diffeomorphisms $\phi_i$ from
$Y\backslash S(r)$ into $M$, containing $M\backslash S_i(3r)$, where
$S_i(3r) = \cup_{j=1}^L B_{t_i}(p_j^i,3r)$, such that $\phi_i^*g_i$
converges smoothly to $\bar{g}$. That means the curvatures $|\rem|(\cdot,t_i)$
are uniformly bounded on $M\backslash S_i(3r)$. We may assume the balls
$B_{t_i}(p_j^i,r/2)$ are disjoint for $i$ big enough. Since the reason for the 
formation of the singular set $S$ is the curvature blow up,
we can find points $q_j^i\in B_{t_i}(p_j^i,3r)$ so that 
$\max_{B_{t_i}(q_j^i,4r)\times [-r^2,r^2]} 
|\rem|(x,t_i + s) = |\rem|(q_j^i,t_i) := Q_j^i$ (at least for $i$
big enough). If that maximum is attained at some time $s \neq 0$,
we can just replace $t_i$ by $t_i +s$  and continue the consideration.
We may assume $Q_j^i \to \infty$ as $i\to\infty$ for every $j$,
since otherwise $S_i(3r) \cap B_{t_i}(q_j^i,4r)$ would at the same time give rise to a
part of the  singular set $S$ and also converge to a smooth part of $Y$ (due
to the uniform curvature bounds), which is not possible.

\begin{claim}
\label{claim-pi}
There exists a uniform constant $C_1$, so that for every $r > 0$,
every sequence of points $q_i \in M$ and a sequence of times $t_i \to \infty$,
with the property that
$\max_{B_{t_i}(q_i,r)\times [-r^2,r^2]}|\rem|(x,t_i+s) = |\rem|(q_i,t_i)$,
we have that for every $0 < \rho \le r$ there exists  an $i_0$ with the property  
\begin{equation}
\label{eq-equiv-int}
\int_{B_{t_i}(q_i,\rho)} |\rem|^{\frac{n}2}(t_i+s)\, dV_{t_i+s} \le 
C_1 \int_{B_{t_i}(q_i,\rho)} |\rem|^{\frac{n}2}\, dV_{t_i},
\end{equation}
for all $s\in [-\rho^2,\rho^2]$ and all $i\ge i_0$. 
\end{claim}

\begin{proof}
Assume the claim were not true. Then for every $j$ there would exist an $0 < \rho_j \le r_j$,
so that for every $i$ there would exist  $k_{ij} > i$ with the property,
\begin{equation}
\label{eq-bad}
\int_{B_{t_{k_{ij}}}(q_{k_{ij}},\rho_j)}|\rem|^{\frac{n}{2}}\, dV_{t_{k_{ij}}} 
\le \frac{1}{j}\int_{B_{t_{k_{ij}}}(q_{k_{ij}},\rho_j)} 
|\rem|^{\frac{n}{2}}(t_{k_{ij}}+s_{k_{ij}})\, dV_{t_{k_{ij}}+s_{k_{ij}}},
\end{equation}
for some $s_{k_{ij}} \in [-r_j^2,r_j^2]$. Since $Q_i\to \infty$, 
we may choose a subsequence $k_{ij}$  
so that $\rho_j \ge (Q_{k_{ij}})^{-1}$. If we define a sequence 
of rescaled metrics $\tilde{g}_{k_{ij}}(\tau) = 
Q_{k_{ij}}g(t_{k_{ij}}+\tau (Q_{k_{ij}})^{-1})$, 
(\ref{eq-bad})
can be rewritten as
\begin{equation}
\label{eq-bad1}
\int_{B_{\tilde{g}_{k_{ij}}(0)}(q_{k_{ij}},r_j(Q_{k_{ij}})^{-1/2})}
|\rem|^{\frac{n}{2}}\, dV_{\tilde{g}_{k_{ij}}(0)} \le \frac{C}{j},
\end{equation}
since $\int_M|\rem|^{\frac{n}{2}}\, dV_t \le C$, uniformly along the flow. 
The pointed sequence of solutions $(B_{\tilde{g}_{k_{ij}}(\tau)}
(q_j^{k_{ij}},r_j\sqrt{Q_{k_{ij}}}), \tilde{g}_{k_{ij}}(\tau), q_{k_{ij}})$
converge to a complete solution $(X,\tilde{g}_{\infty}(\tau), q_{\infty})$ with 
$|\rem|_{\tilde{g}_{\infty}}(q_{\infty},0) = 1$. Therefore taking the limit as $i, j\to\infty$ in 
(\ref{eq-bad1}) yields a contradiction since
$$0 \neq \int_X |\rem|^{\frac{n}{2}}\, dV_{\tilde{g}_{\infty}} \le \liminf_{i,j\to\infty}
\int_{B_{\tilde{g}_{k_{ij}}(0)}(q_{k_{ij}},r_j\sqrt{Q_{k_{ij}}})}
|\rem|^{\frac{n}{2}}\, dV_{\tilde{g}_{k_{ij}}(0)} = 0.$$
\end{proof}

Take $C_1$ as in the Claim \ref{claim-pi}. Take $\epsilon > 0$ 
small so that $C_1\epsilon < \epsilon_0$,
where $\epsilon_0$ is taken from Proposition \ref{lemma-lemma_regularity}.
If $r > 0$ and a sequence $\{q_j^i\}$ (for $1 \le j \le L$) are as we have
constructed in the paragraph just prior to Claim \ref{claim-pi} (depending on the
chosen $\epsilon$), then we have the following:

\begin{claim}
\label{claim-bad-set}
There exists $i_0$ so that 
$$\int_{B_{t_i}(q_j^i,r)}|\rem|^{\frac{n}{2}}\, dV_{t_i} > \epsilon, \qquad \mbox{for}
\,\, i\ge i_0,$$
for all $1\le j \le L$.
\end{claim}

\begin{proof}
Assume the claim is not true. Then for some $j\in \{1,\dots, L\}$, by Claim \ref{claim-pi},
we have
$$\int_{B_{t_i}(q_j^i,r)}|\rem|^{\frac{n}{2}}\, dV_{t_i+s} \le C_1\epsilon < \epsilon_0,$$
for all $s\in [-r^2,r^2]$. By Proposition \ref{lemma-lemma_regularity},
$$|\rem|(q_j^i,t_i) \le \frac{C}{r^2}, \qquad \mbox{for} \,\, i\ge i_0,$$
which contradicts $Q_j^i\to\infty$ as $i\to\infty$.
\end{proof}

The previous discussion and Claim \ref{claim-bad-set} imply that there is an
$i_0$ so that for every $\bar{p}_j$, where $1 \le j \le L$, we can find a 
sequence of points $\{q_j^i\}_{i\ge i_0} \in M$ so that
\begin{equation}
\label{eq-bigger-e}
\int_{B_{t_i}(q_j^i,2r)}|\rem|^{\frac{n}{2}}\, dV_{t_i} > \epsilon,
\end{equation}
for all $i \ge i_0$.  By our careful choices of $r$ and the covering of $S$,
(by increasing the constant $1000$ in the number of balls of radius $2r$
covering $S$, if necessary), for every fixed $i \ge i_0$,
the number of points $q_j^i \in L_i^r$ with the property that the balls $B_{t_i}(q_j^i,r)$ are
disjoint is at least $[\frac{\tilde{C}}{\epsilon}] + 2$,  which contradicts (\ref{eq-upper-N}).
This finishes the proof of Proposition \ref{prop-finite}.  
\end{proof}

Combining Proposition \ref{prop-finite} and  the results from \cite{natasa_KR},
so far we have proved the following: if $g(t)$ is the K\"ahler-Ricci flow, with
$|\ric|(\cdot,t) \le C$ and $\int_M|\rem|^{\frac{n}{2}}\, dV_t \le C$,
uniformly along the flow, then for every sequence of times $t_i\to\infty$
there exists a subsequence so that $(M,g(t_i+t))$ converges to
$(M_{\infty},g_{\infty}(t))$ in the following sense. There are finitely many 
points $p_1,\dots,p_N\in M_{\infty}$ and $\{p_j^i\} \in M$, for $1\le j\le N$,
so that for every $r > 0$ there are diffeomorphisms $\phi_i$
from $M_{\infty}\backslash \cup_{j=1}^N B_{\infty}(p_j,r)$ into $M$,
with the image of $\phi_i$ containing $M\backslash \cup_{i=1}^NB_{t_i}(p_j^i,2r)$,
and $\phi_i^* g(t_i+t)$ smoothly converging to the K\"ahler-Ricci soliton $g_{\infty}(t)$,
outside the singular points. This in particular means $M_{\infty}$ is smooth
outside finitely many points. We would like to understand the structure of
those singular points.  

By Fatou's lemma,
$$\int_{M_{\infty}}|\rem|^{\frac{n}{2}}\, dV_{\infty} \le C < \infty.$$
By the continuity of volume under the condition of the lower bound on Ricci curvatures
of a sequence of manifolds (see \cite{C}), we have
$$\lim_{i\to\infty} \vol_{t_i}(B_{t_i}(x_i,r)) = \vol_{\infty}(B_{\infty}(x_{\infty},r)),$$
for every sequence $x_i\in M$ and $r > 0$ such that a sequence of balls
$B_{t_i}(x_i,r)$ converges in Gromov-Hausdorff topology, as $i\to\infty$, to 
a ball $B_{\infty}(x_{\infty},r)$. This together with (\ref{eq-vol-upper}) yield
$$\vol_{\infty}(B_{\infty}(x,r)) \le 2C_1 r^n, \qquad \mbox{for all} \,\,x\in M_{\infty}.$$
Let $\epsilon << \epsilon_0$. The previous estimate implies there is an $r_0$
so that for $r \le r_0$, for every $x\in M_{\infty}$,
$$\int_{B_{\infty}(x,r)}|\rem|^{\frac{n}{2}}\, dV_{g_{\infty}(t)} < \epsilon,$$
for all $t\in [-2r^2,2r^2]$. Fix some $r > 0$ and denote by  $D_i^r = M\backslash
(\cup_{j=1}^N B_{t_i}(p_j^i,r))$. By the definition of convergence,
there exists an $i_0$ so that for $i \ge i_0$, all $x\in D_i^{2r}$ and
$t\in [-r^2,r^2]$,
$$\int_{B_{t_i}(x,r)}|\rem|^{\frac{n}{2}}\, dV_{t_i + t} < 2\epsilon < \epsilon_0.$$ 
By Proposition \ref{lemma-lemma_regularity} we have
$$\sup_{[t_i-\delta r^2,t_i+\delta r^2]\times D_i^{2r}}|\rem| \le \frac{C}{r^2},$$
for uniform constants $C$ and $\delta$. By Shi's estimates,
$$\sup_{[t_i-\delta_1 r^2,t_i+\delta_1 r^2]\times D_i^{2r}}|D^k\rem| \le C(k,n,r),$$
where we can assume without no loss of generality that $\delta_1 = 1$.
We can extract a subsequence, so that $(D_i^{2r},g(t_i+t))$
converges to a smooth solution to the Ricci flow, $(D_{\infty}^{2r},g_{\infty}(t))$,
for $t\in [-r^2,r^2]$. As in \cite{se_ti} and \cite{natasa_KR}, we can show
$g_{\infty}(t)$ satisfies the K\"ahler Ricci soliton equation,
\begin{eqnarray}
\label{equation-soliton-limit}
& & \ric(g_{\infty}) + \nabla\bar{\nabla}f_{\infty} - g_{\infty} = 0, \\
& & \nabla\nabla f_{\infty} = 0. \nonumber
\end{eqnarray}

We now choose a sequence $\{r_l\}\to 0$ with $r_{l+1} < r_l/2$ and
perform the above construction for every $l$. If we set $D_i(r_l) =
\{x\in M|x\in D_i^{r_j}$, for some $j\le l\}$ then we
have
$$D_i(r_l)\subset D_i(r_{l+1}) \subset\dots \subset M.$$
For each fixed $r_l$, by the same arguments as above, each sequence
$\{D_i(r_l),g(t_i+t)\}$ for $t\in [-r_l^2,r_l^2/2]$, has a smoothly
convergent subsequence to a smooth limit $D(r_l)$ with a metric
$g_{\infty}^{r_l}$, satisfying the K\"ahler Ricci soliton condition.
We can now set $D = \cup_{l=1}^{\infty}D(r_l)$ with the induced metric
$g_{\infty}$ that coincides with $g^{r_l}$ on $D(r_l)$ and which is
smooth on $D$.

Following section $5$ in \cite{anderson1989} we can show there are
finitely many points $\{x_i\}$ so that $M_{\infty} = D\cup \{x_i\}$ is
a complete length space with a length function $g_{\infty}(0)$, which
restricts to the K\"ahler Ricci soliton on $D$. By
(\ref{equation-soliton-limit}), since $|\ric(g_{\infty}(t))| \le C$,
we get,
$$\sup_{M_{\infty}\backslash \{x_1,\dots,x_N\}}|D^2 f_{\infty}| \le \tilde{C}.$$
Similarly as in \cite{se_ti} we get $|\nabla f_{\infty}|, |f_{\infty}|$
are uniformly bounded on $M_{\infty}\backslash \{x_1,\dots,x_N\}$.

We will include below the computation from \cite{natasa2005}, that
holds for all $x\in M_{\infty}\backslash \{p_1,\dots,p_N\}$.
Denote $\rem(g_{\infty})$ shortly by $\rem$.  By Bochner-Weitzenbock
formulas we have
\begin{equation}
\label{equation-bochner}
\Delta|\rem|^2 = -2\langle\Delta \rem,\rem\rangle + 2|\nabla\rem|^2
- \langle Q(\rem),\rem\rangle,
\end{equation}
where $Q(\rem)$ is quadratic in $\rem$. The Laplacian of the
curvature tensor in the K\"ahler case reduces to
$$\Delta R_{i\bar{j}k\bar{l}} = \nabla_i\nabla_{\bar{l}}R_{\bar{j}k} + 
\nabla_{\bar{j}}\nabla_kR_{i\bar{l}} + S_{i\bar{j}k\bar{l}},$$ 
where $S(\rem)$ is quadratic in $\rem$. Since our metric $g_{\infty}$ is the
soliton metric $g_{\infty}$, satisfying (\ref{equation-soliton-limit}) outside
$\{p_1,\dots,p_N\}$, by commuting the
covariant derivatives, we get
\begin{eqnarray}
\label{equation-simpler}
\Delta R_{i\bar{j}k\bar{l}} &=& (f_{\infty})_{\bar{j}k\bar{l}i} +(f_{\infty})_{i\bar{l}k\bar{j}} +
S_{i\bar{j}k\bar{l}} \nonumber \\
&=& (f_{\infty})_{\bar{j}\bar{l}ki} + \nabla_i(R_{\bar{j}k\bar{l}m}(f_{\infty})_m) +
(f_{\infty})_{ik\bar{l}\bar{j}} + \nabla_{\bar{j}}(R_{i\bar{l}k\bar{m}}(f_{\infty})_{\bar{m}}) 
+ S_{i\bar{j}k\bar{l}} \nonumber \\
&=& \nabla_i(R_{\bar{j}k\bar{l}m})(f_{\infty})_m + \nabla_{\bar{j}}
(R_{i\bar{l}k\bar{m}})(f_{\infty})_{\bar{m}} + S_{i\bar{j}k\bar{l}} \nonumber \\
&=& \nabla\rem * \nabla (f_{\infty}),
\end{eqnarray} 
where we have effectively used the fact that $(f_{\infty})_{ij} =
(f_{\infty})_{\bar{i}\bar{j}} = 0$ and $A*B$ denotes any tensor
product of two tensors $A$ and $B$ when we do not need precise
expressions.  By using that $|f_{\infty}|_{C^1(M_{\infty}\backslash
  \{p_1,\dots,p_N\}} \le C$, and identities
(\ref{equation-bochner}) and (\ref{equation-simpler}) we get,
$$\Delta|\rem|^2 \ge -C|\nabla\rem||\rem| + 2|\nabla\rem|^2 -
C|\rem|^3.$$
By interpolation inequality we have
\begin{eqnarray*}
\Delta|\rem|^2 &\ge& (2-\theta)|\nabla\rem|^2 - C(\theta)|\rem|^2 -
C|\rem|^3 \\
&\ge& (2-\theta)|\nabla|\rem||^2 - C(\theta)|\rem|^2 - C|\rem|^3,
\end{eqnarray*}
for some small $\theta$. Also,
$$\Delta|\rem|^2 = 2\Delta|\rem||\rem| + 2|\nabla|\rem||^2,$$
and therefore,
\begin{equation}
\label{equation-curv}
\Delta|\rem||\rem| \ge -\theta/2|\nabla|\rem||^2 - C(\theta)|\rem|^2
- C|\rem|^3.
\end{equation}

By the same arguments as in \cite{natasa2005} (see sections $5$ and $6$),
we can show
\begin{enumerate}
\item [(a)] $M_{\infty} = D\cup \{p_1,\dots,p_N\}$ is a complete
orbifold with isolated singularities $\{p_1,\dots,p_N\}$.
\item[(b)] 
A limit metric $g_{\infty}$ on $D$ can be extended to an
orbifold metric on $M_{\infty}$ (denote this extension by $g_{\infty}$
as well). More precisely, in an orbifold lifting around singular
points, in an appropriate gauge, the K\"ahler Ricci soliton metric
$g_{\infty}$ can be smoothly extended over the origin in a ball in
$\mathrm{C}^{n/2}$.
\end{enumerate}
This finishes the proof of Theorem \ref{theorem-convergence}.  
\end{proof}

\section{Convergence of plurianticanonical divisors}
\label{section-tian}

We will prove Theorem \ref{thm-curv-bound} by contradiction.
Assume there is a sequence $t_i\to\infty$ such that $Q_i := \max_M|\rem|(\cdot,t_i) \to\infty$
as $i\to\infty$. Considering the sequence of solutions $(M,g(t_i+t))$ and using 
the ideas of Tian, we will show that a subsequence $(M,g(t_i))$ smoothly converges to a 
smooth K\"ahler manifold $(M_{\infty},g_{\infty})$, which will contradict the fact $Q_i\to\infty$.
By Theorem \ref{theorem-convergence}, we may assume $(M,g(t_i+t))$ converges
to a K\"ahler Ricci soliton $(M_{\infty},g_{\infty})$ in an orbifold sense, with
finitely many orbifold points $p_1, p_2, \dots p_l$. 

A line bundle $E$ on $M_{\infty}$ is a line bundle on the regular part $M_{\infty}'$
such that for each local lifting of a singular point $x$,  $\pi_x: \tilde{B}_x\to M_{\infty}$, 
the pullback $\pi_x^*E$ on $\tilde{B}_x\backslash \pi^{-1}(x)$ can be extended to the
whole $\tilde{B}_x$. We will consider plurianticanonical line bundles $K_{M_{\infty}}^{-m}$
for $m\in\mathbb{N}$. A golbal section of $K_{M_{\infty}}^{-m}$ is an element in
$H^0(M_{\infty}',K_{M_{\infty}}^{-m})$, which can be extended across the sigular points
in the above sense. Therefore, $H^0(M_{\infty},K_{M_{\infty}}^{-m})$ can be understood
as a linear space of all global sections of $K_{M_{\infty}}^{-m}$ and $g_{\infty}$ induces
a hermitian orbifold metric on $K_{M_{\infty}}^{-m})$ .

The proof of Theorem \ref{thm-curv-bound} will follow from a sequence of 
lemmas and claims. The following Lemma is taken from \cite{Ti} and \cite{tian1990}.

\begin{lemma}
\mbox{}
\label{lem-section-1}
\begin{enumerate}
\item[(i)]
Let $S^i$ be a global holomorphic section of $H^0(M,K_M^{-m})$, 
$\int_M ||S^i||_{g(t_i})\,dV_{t_i} = 1$, where $||\cdot||_{g(t_i)}$ is the hermitian metric 
of $K_M^{-m}$ induced by $g(t_i)$. Then there is a subsequence 
$\{S^i\}$ converging to a global holomorphic section $S^{\infty}$ in 
$H^0(M_{\infty},K_{M_{\infty}}^{-m})$. 
\item[(ii)]
Any section $S$ in $H^0(M_{\infty},K_{M_{\infty}}^{-m})$ is the limit of some
sequence $\{S^i\}$ where $S^i\in H^0(M,K_M^{-m})$.
\end{enumerate}
In particular, (i) and (ii) imply the dimension of $H^0(M,K_M^{-m})$ is the same as that
of $H^0(M_{\infty},K_{M_{\infty}}^{-m})$.
\end{lemma}
  
Sections $S^i$ and $S_{\infty}$ are not  on the same K\"ahler manifolds. From the definition
of the convergence, for every compact set $K\subset M_{\infty}$ there are 
diffeomorphisms $\phi_i$ from compact subsets $K_i\subset M$onto $K$ so that
$(\phi_i^{-1})^*g(t_i) \to g_{\infty}$ and $\phi_i^*\circ J_i\circ (\phi^{-1}_i)^* \to J_{\infty}$,
as $i\to\infty$. The convergence of sections in the previous lemma
means that the sections $\phi_{i*}(S^i)$ converge to a section $S^{\infty}$ of
$K_{M_{\infty}}^{-m}$ in the $C^{\infty}$-topology.
  
\begin{proof}[Proof of Lemma \ref{lem-section-1}]
Let $S^i$ be as in the statement of the lemma. If $\Delta_i$ is the laplacian
and $||\cdot||_i$ the inner product with respect to metric $g(t_i)$, 
by a direct computation, we have
$$\Delta_i(||S^i||_i^2)(x) = ||D_iS^i||_i^2(x) - \ric(S^i,S^i)(x),$$
where $D_i$ is the covariant derivative with respect to $g(t_i)$. Since $\ric$
is uniformly bounded along the flow,
\begin{equation}
\label{eq-moser-S}
\Delta_i(||S^i||_i^2)(x) \ge -C||S^i||_i^2(x),
\end{equation}
Since we have $\int_M ||S^i||_i\, dV_{t_i} = 1$ and Sobolev inequality 
(\ref{equation-Sobolev}) which holds for all $t\ge 0$, 
with the uniform upper bound on the Sobolev constant, applying Moser's iteration
to (\ref{eq-moser-S}), there is a uniform constant $C = C(n)$ such that
\begin{equation}
\label{eq-S-bound}
\sup_M (||S^i||^2_i(x)) \le C, \qquad \mbox{for all} \,\, i .
\end{equation}
Once we have the estimate (\ref{eq-S-bound}) we can proceed as in the proof of 
Lemma $2.1$ in \cite{Ti}. We need to show that for
every integer $j > 0$, the $j$-th covariant derivatives of $\phi_{i*}S^i$ ($\phi_{i*}$ is a
diffeomorphism that comes from a definition of convergence) are bounded 
in every compact set $K\subset M_{\infty}\backslash\{p_1,\dots, p_l\}$. Depending
only on $K$ there is an $r > 0$ so that for every $x\in K$ the geodesic ball $B_r(x,g(t_i))$ is
uniformly biholomorphic to an open subset of $\mathbb{C}^k$. On each $B_r(x,g(t_i))$,
the section $S^i$ is represented by a holomorphic function. We can use well-known Cauchy
integral formula to get uniform bounds on $j$-th covariant derivatives of $S^i$.

Absolutely the same proof as that of Lemma $2.2$ works also in our case to prove that any 
section $S$ in $H^0(M_{\infty},K_{M_{\infty}}^{-m})$ is the limit of some sequence $\{S^i\}$
with $S^i$ in $H^0(M,K_M^{-m})$. This in particular implies that the dimension of
$H^0(M_{\infty},K_{M_{\infty}}^{-m})$ is the same as that of $H^0(M, K_M^{-m})$.
\end{proof}

\section{Smoothing property in complex dimensions $\ge 3$}
\label{section-smoothing}

Given a complex manifold $X$ with strongly pseudoconvex boundary Y, we define
$\mathcal{B}^{p,q}(Y)$ to be the space of smooth sections of $\Lambda^{p,q}(X)
\cap \Lambda^{p,q}(T_{\mathbb{R}}^*Y\otimes \mathbb{C})$. The $\bar{\partial}$-
operator of $X$ induces  $\bar{\partial}_b: \mathcal{B}^{p,q}(Y) \to \mathcal{B}^{p,q+1}(Y)$.
Let $\bar{\partial}^*_b$ be the adjoint operator of $\bar{\partial}_b$ on $Y$. 
Since $\bar{\partial}_b^2 = 0$, we have the boundary complex
$$0 \to \mathcal{B}^{p,0} \stackrel{\bar{\partial}_b}{\to} \mathcal{B}^{p,1} \to 
\dots\stackrel{\bar{\partial}_b}{\to} \mathcal{B}^{p,n-1} \to 0.$$
The cohomology of this boundary complex is called Kohn-Rossi 
cohomology and is denoted by $H^{p,q}(\mathcal{B})$. 

By Theorem \ref{theorem-convergence} we have that a sequence $(M,g(t_i))$ 
converges to a K\"ahler Ricci soliton in an orbifold sense. In other words, there are
points $p_{1i},\dots p_{Ni}$ in $(M,g(t_i))$ and $p{1_{\infty}},\dots, p_{N{\infty}}$
in $M_{\infty}$ so that: for every $r > 0$ there are diffeomorphisms $\phi_i$
from compact sets $M_{\infty}\backslash (B_r(p_{1\infty})\cup\dots \cup
B_r(p_{N\infty})\}$ into $M$, containing 
$M\backslash B_{2r}(p_{1i},g(t_i))\cup\dots\cup B_{2r}(p_{Ni},g(t_i))\}$ so that
$\phi_i^*g_i$ and $\phi_i^*\circ J_i\circ (\phi_i^{-1})^*$ converge to $g_{\infty}$
and $J_{\infty}$, respectively. We would like to understand the holomorphic
structure of $B_r(p_{il})$, $1\le l \le N$, for sufficiently small $r$ and big $i$.
The same problem arose in \cite{Ti} where the sequence of K\"ahler-Einstein
manifolds converges to a K\"ahler-Einstein orbifold with isolated singularities.
The main tool there was Kohn's estimate for $\Box_b$-operators that works only
when $k \ge 3$.  Let $S_{\infty r} := \partial B_r(p_{\infty})$ (where $p_{\infty}$
is one of the points $\{p_{1\infty},\dots, p_{N\infty}\}$) be the level surface
of the distance function $\rho_{\infty}(\cdot,p_{\infty})$. The Levi form on
$S_{r\infty}$ is
$$(L_1, L_2) := 2(\partial\bar{\partial}\rho_{\infty}(\cdot,p_{\infty}), L_1\wedge
\bar{L}_2).$$
It is positive definite for $r$ small since $\rho(\cdot, p_{\infty})$ is convex
near $p_{\infty}$. Let $S_{ir}$ be the level 
surface $\{x \in (M, g(t_i)) | \rho_{\infty}(p_{\infty}, \phi_i^{-1}(x)) = r\}$,
which is also a smooth, pseudoconvex manifold. Denote by $\tilde{S}_i$
the universal covering of $S_{ir}$, it is diffemorphic to $S^{n-1}$. By using
the result in (\cite{Y}) that $H^{0,1}(\mathcal{B}(S^{n-1}) = 0$ for $n \ge 6$
(or $k \ge 3$), Tian (\cite{Ti}) showed the following Kohn's estimate for
$\bar{\partial}_b$-operator,
$$C||u||_2^2 \le ||\bar{\partial}_bu||_2^2 + ||\bar{\partial}^*_b u||_2^2,$$
for any $u \in \mathcal{B}^{0,1}(\tilde{S}_i)$ and a uniform constant $C$.
If $\lambda$ is the smallest eigenvalue of $\Box_b = \bar{\partial}_b
\bar{\partial}_b^* + \bar{\partial}^*_b\bar{\partial}_b$, the following estimate 
is equivalent to $\lambda \ge c > 0$. Tian used this to show there exist
embeddings $k_{ij}: \tilde{S}_j \to \mathbb{C}^k$ so that $k_{ij}(\tilde{H}_j)$
converge to $S^{n-1}$ as submanifolds of $\mathbb{C}^k$, in sufficiently
nice topology. Following the same proof as that of Tian in \cite{Ti}
we get the following proposition.

\begin{proposition}
\label{prop-smooth-limit}
Let $(M,g(t))$ be the K\"ahler Ricci flow and $t_i\to\infty$ as at the beginning
of section \ref{section-tian}. Then either a sequence $\{(M, g(t_i))\}$ converges 
to a K\"ahler-Ricci soliton in $C^p$-topology, or there is a smooth K\"ahler-Ricci
soliton $(M_{\infty},g_{\infty})$ so that a sequence $\{(M,g(t_i))\}$ converges
to $(M_{\infty},g_{\infty})$ in the $C^p$-topology, outside finitely many points.  
\end{proposition}

The proof of the above proposition uses results from section \ref{section-tian},
more precisely Lemma \ref{lem-section-1},
to get the orthonormal bases $\{S_j^i\}_{j=0}^{N_m}$ of $H^0(M,K_M^{-m})$,
with respect to metric $g(t_i)$, that converges to the basis 
$\{S_0^{\infty},\dots,S_{N_m}^{\infty}\}$ of $H^0(M_{\infty}, g_{\infty})$
(defining the Kodaira's embedding of $M_{\infty}$ into $\mathbb{CP}^{N_m}$).
In particular, for $j$ sufficiently large, these $\{S_j^i\}$ give embeddings of
$M$ into $\mathbb{CP}^{N_m}$. The proof of Proposition \ref{prop-smooth-limit}
can be found in section $3$ of \cite{Ti}.

\begin{proof}[Proof of Theorem \ref{thm-curv-bound}]
The proof is by contradiction. Assume there is a sequence $t_i\to\infty$
so that $\max_M|\rem|(\cdot,t_i) \to \infty$ as $i\to\infty$. 
Take $\eta > 0$ small. Since $|\ric| \le C$ along the flow, 
for every $i$, the metrics $g(t_i+t)$ are uniformly equivalent for $t\in [-\eta,\eta]$,
with constants independent of $i$, that is,
$$cg(t_i+t) \le g(t_i) \le Cg(t_i+t), \qquad \mbox{for all} \,\, i \,\, \mbox{and all}
\,\, t\in [-\eta,\eta].$$  
This means that by Proposition \ref{prop-smooth-limit}, either a sequence
of solutions $\{(M,g(t_i+t))\}$ converges to a K\"ahler-Ricci soliton solution in
$C^p$-topology, or there is a smooth K\"ahler-Ricci soliton 
$(M_{\infty},g_{\infty}(t))$ so that a sequence $\{(M,g(t_i+t))\}$ converges 
to $(M_{\infty},g_{\infty}(t))$ in the $C^p$-topology, outside finitely many
points, for every $t\in [-\eta,\eta]$. This yields each $B_r(p_{l\infty},g_{\infty}(t))$,
for $t\in [-\eta,\eta]$,  with small $r$, is a smooth ball in $\mathbb{C}^k$. 
Therefore, for $i$ large enough and $t\in [-\eta,\eta]$,
$B_r(p_{li},g(t_i+t))$ are smooth balls in $\mathbb{C}^k$ as well.

To finish the proof of Theorem \ref{thm-curv-bound} we will need one of
the achievements of Perelman, known as the {\em pseudolocality theorem}
for the Ricci flow (\cite{perelman2002}). It says that 
for every $\alpha > 0$ there exists $\delta > 0$, $\epsilon > 0$ with
the following property. Suppose we have a smooth solution to the Ricci
flow $(g_{ij})_t = -2R_{ij}$ for $0 \le t \le (\epsilon r_0)^2$ and
assume that at $t = 0$ we have $R(x) \ge -r_0^{-2}$ and $\vol(\partial
\Omega)^n \ge (1 - \delta)c_n\vol(\Omega)^{n-1}$ for any $x$ and
$\Omega \subset B(x_0,r_0)$, where $c_n$ is the euclidean isoperimetric
constant. Then we have an estimate $|\rem|(x,t) \le \alpha t^{-1} +
(\epsilon r_0)^{-2}$, whenever $0 < t \le (\epsilon r_0)^2$, $d(x,t) =
\dist_t (x,x_0) \le \epsilon r_0$.

We will  apply the pseudolocality theorem to each of the 
solutions $\tilde{g}_i(t) = g(t_i - \eta + t)$, for $i\ge i_0$. Fix $\alpha > 0$ and choose 
$\epsilon, \delta > 0$ as in Perelman's result. Since for
$r_0 = r$ the conditions of the pseudolocality theorem are
satisfied,  $|\rem(\tilde{g}_i)|(x,t) \le \alpha t^{-1} + (\epsilon r)^{-2}$,
whenever $0 < t < (\epsilon r)^2$ and $\dist_{\tilde{g}_i(t)}(x,p_{li}) 
\le \epsilon r$. If $(\epsilon r)^2 < \eta$, apply Perelman's result to
$\tilde{g}_i(\cdot,t)$ starting at $t = (\epsilon r)^2/2$, 
since it also satisfies the assumptions of his theorem. 
This will help us extend our uniform in $i$ estimates on $|\rem|(\tilde{g}_i(t))$
past time $(\epsilon r)^2$. Repeat this procedure until we hit
$t = \eta$. To summarize, we get that
$$\sup_M |\rem|(\cdot,t_i) \le C(n,\epsilon,\alpha,\eta,r), \qquad \mbox{for all} \,\,\,
i \ge i_0,$$
and we get a contradiction with $\max_M|\rem|(\cdot,t_i) \stackrel{i\to\infty}{\to} 
\infty$.
\end{proof}

\end{document}